\documentclass[a4paper,11pt]{amsart}
\usepackage{amsmath,amssymb,amsfonts,amsthm,exscale,calc}
\usepackage[latin1]{inputenc}
\usepackage{latexsym,textcomp}
\usepackage{graphicx,epsf}
\usepackage{dsfont}
\usepackage[german, english]{babel}
\usepackage{hyperref}

\newtheorem{lemma}{Lemma}[section]
\newtheorem{theorem}[lemma]{Theorem}

\newtheorem{proposition}[lemma]{Proposition}
\newtheorem{corollary}[lemma]{Corollary}

\theoremstyle{definition}

\newtheorem*{remark}{Remark}

\numberwithin{equation}{section}
\newcommand{\comment}[1]{}

\newcommand{\R}{{\mathbb R}}
\newcommand{\C}{{\mathbb C}}

\newcommand{\N}{{\mathbb N}}

\newcommand{\Deg}{{\mathrm {Deg}}}

\newcommand{\Hm}[1]{\leavevmode{\marginpar{\tiny%
$\hbox to 0mm{\hspace*{-0.5mm}$\leftarrow$\hss}%
\vcenter{\vrule depth 0.1mm height 0.1mm width \the\marginparwidth}%
\hbox to 0mm{\hss$\rightarrow$\hspace*{-0.5mm}}$\\\relax\raggedright
#1}}}

\begin{document}

\title[Short time behavior of selfadjoint semigroups]{Note on short time behavior of semigroups associated to selfadjoint operators}

\author[Keller]{Matthias Keller}
\address{M. Keller, Mathematisches Institut \\Friedrich Schiller Universit{\"a}t Jena \\07743 Jena, Germany }\email{m.keller@uni-jena.de}

\author[Lenz]{Daniel Lenz}
\address{D. Lenz, Mathematisches Institut \\Friedrich Schiller Universit{\"a}t Jena \\07743 Jena, Germany } \email{daniel.lenz@uni-jena.de}

\author[M\"unch]{Florentin M\"unch}
\address{F. M\"unch, Mathematisches Institut \\Friedrich Schiller Universit{\"a}t Jena \\07743 Jena, Germany }
\email{Florentin.Muench@uni-jena.de}

\author[Schmidt]{Marcel Schmidt}
\address{M. Schmidt, Mathematisches Institut \\Friedrich Schiller Universit{\"a}t Jena \\07743 Jena, Germany } \email{schmidt.marcel@uni-jena.de}

\author[Telcs]{Andras Telcs}
\address{A. Telcs, Department of Computer Science and Information Theory, Budapest University of Technology and Economics,
Magyar tudosok korutja 2, H-1117, Budapest, Hungary}
\email{telcs@szit.bme.hu}

\begin{abstract} We present a simple observation showing
that the heat kernel on a locally finite  graph behaves for short
times $t$ roughly like $t^d$,  where $d$ is the combinatorial
distance. This is very different from the classical Varadhan type
behavior on manifolds. Moreover, this also gives that  short time
behavior and  global behavior of the heat kernel are governed by two
different metrics whenever the degree of the graph is not uniformly
bounded.
\end{abstract}


\maketitle



\section*{Introduction}
There is a strong connection between  basic geometric features  of a
Riemannian manifold and the behavior of the heat kernel of the
diffusion  associated to its   Laplace Beltrami operator. In
particular, global features of the heat diffusion  like e.g.
stochastic completeness,  recurrence, Liouville properties and
bounds on the infimum of the (essential) spectrum  are captured by
asymptotic behavior of the volume growth of balls
\cite{Broo,KarpLi,Gri}. Likewise, short time behavior of the heat is
related to distance of points. Indeed, there are famous results
giving \textit{Varadhan type behavior}, i.e.,
$$ \lim_{t\to 0+} t \log p_t (x,y) = - \frac{\varrho^2 (x,y)}{2}$$
with $\varrho$ being the geodesic distance between and $p$ denoting
the heat kernel \cite{Var,Var2}.

Starting with the seminal work of Sturm \cite{Stu}  it has became
clear that similar results hold in more generality; namely,  in the
context of strongly local Dirichlet forms. There geometry is
captured by the so-called intrinsic metric. In particular, there are
volume criteria for stochastic completeness, recurrence, Liouville
properties as well as bounds on the infimum of the (essential)
spectrum \cite{Not,Stu}. Moreover, in a series of papers
\cite{AH,ERS,ERSZ1,ERSZ2,HR, Nor} Varadhan type  behavior could be shown for rather
general strongly local Dirichlet forms. See also \cite{GT,T} for related results on metric measure spaces and fractal spaces.

Let us emphasize that  in both the volume criteria capturing the
global situation and the short time Varadhan behavior  the geometry
enters via the same quantity viz the intrinsic metric.

Recently, a concept of intrinsic metric for general regular
Dirichlet forms has been introduced in \cite{FLW} (see
\cite{Fol1,GHM,Uem} for related material as well). This can in
particular be applied to graphs, i.e., Dirichlet forms on discrete
spaces. In this situation  a particular nice example of such a
metric was then provided by Huang in \cite{Hua}. It turns out that
the combinatorial metric on a graph is an intrinsic metric if and
only if the degree is uniformly bounded \cite{HKMW,KLWS}. Thus, for
graphs with unbounded degree intrinsic metrics are rather different
from combinatorial metrics.

For graphs  intrinsic metrics  have proven extremely useful in
capturing global features, see e.g.  \cite{FLW,Fol1,GHM,HuK,HKW}.
Thus, one might expect that they  also allow  one to recover a
Varadhan type result. The aim of this note is to show that this is
NOT  the case. Indeed, our main result shows that, for small $t>0$,
the heat kernel $p_t$ on a locally finite graph satisfies
$$(*)\;\:\; |p_t (x,y) - c(x,y) t^d|\leq C(x,y) t^{d+1}$$
and hence
$$(**)\;\:\;  \lim_{t\to 0+} \frac{ \log p_t (x,y)}{\log t}  = d (x,y),$$
where $d$ is the combinatorial graph distance and $c(x,y)$ and
$C(x,y)$ are suitable positive  constants. This short time behavior
clearly implies
$$\lim_{t\to 0+} t \log p_t (x,y) = 0$$
 and, hence,
failure of the Varadhan type behavior. Thus, the short time behavior
on graphs is rather different from the behavior on manifolds  in two
respects:

\begin{itemize}
\item  The actual scaling is rather  different.
\item  Global behavior and short time behavior are governed by two different metrics,
whenever the degree is not uniformly bounded.
\end{itemize}

The result itself is not at all  hard to prove but a rather
straightforward application of the Taylor theorem and we consider
the simple reasoning  rather a virtue of our approach. Despite its
simple proof   on the conceptual level the result is rather
remarkable. It shows a clear difference between local and non-local
Dirichlet forms:  In the strongly local case the whole geometry is
governed by one metric. In the non-local case different metrics play
a role (if the degree is not uniformly bounded).

\smallskip

The main thrust of our investigation and the preceding discussion
concerns the different short time  scaling of heat kernels  for
graphs compared to the strongly local case. However, it is also
worthwhile to compare our results to existing bounds on heat kernels
on graphs. Of course, there is no shortage of such bounds in the
literature. Among  those close to our work we mention contributions
by Davies \cite{Dav} and Pang \cite{Pan} as well as subsequent
investigations by Metzger / Stollmann \cite{MS} and  Schmidt
\cite{Schm}  (and refer the reader to the references of these works
for further literature). These works present  formidable bounds
holding for all times and positions. Still, for short times none of
these bounds seems explicit enough to provide estimates as given
above. In fact, \cite{Dav,Pan} only deal with upper bounds. The
paper \cite{MS} (see its extension \cite{Schm} as well) on the other
hand features a very direct and elegant approach improving on the
upper bounds of \cite{Dav,Pan} and at the same time gives a lower
bound. From these results a somewhat weaker version of $(**)$ can be
inferred, which still gives $(*)$. Note, however, that  \cite{MS}
requires very strong boundedness conditions on the graph viz uniform
boundedness of the degree and uniform bounds on the weights. In
contrast, our result does not require any boundedness assumptions
but rather works for all locally finite graphs.

In this context it may also be worth pointing out two further  nice
features resulting from the generality and simplicity of our
approach: Firstly,  our result is completely independent of the
underlying measure (speed measure) whereas all earlier results seem
to have dealt with a  normalized speed measure. Secondly, our
formalism hides  cumbersome details on the counting of paths present
in earlier approaches, while basically capturing  the fact that the
occurrence of $d$ jumps in a  short time has probability
proportional to  $t^d$.

\smallskip

Our result is actually a consequence of a more general  result
dealing with  semigroups associated to arbitrary non-negative
self-adjoint operators.  This result is discussed in
Section~\ref{Main-abstract-result}.
The application to graphs
is then given in Section~\ref{Application-to-graphs}.
Our results also give some information on manifolds. This is not of
relevance for the Laplace-Beltrami operator as there  much more is
known. However, it may be of interest as we can also deal with
different operators such as higher order operators. We discuss this
shortly in Section \ref{Remark-on-application-to-manifolds}.

Let us finally point out that both our abstract result and the
application to graphs do not only hold for the semigroup but also
for the unitary group.

\section{The main abstract result}\label{Main-abstract-result}
In this section, we present the main abstract ingredient of our
approach. This ingredient is a very simple observation which
essentially shows that Taylor expansion is possible even  for
unbounded selfadjoint operators (under suitable assumptions on the
involved  elements of the Hilbert space). This does not have to do
with Dirichlet forms at all. Accordingly, it gives results for both
semigroups and unitary groups.

\bigskip

We start by recalling some basic consequences of the spectral
theorem for selfadjoint operators in Hilbert space. Let $L$ be a
selfadjoint operator in the complex   Hilbert space $\mathcal{H}$.
By the spectral theorem, we then obtain for any $h\in \mathcal{H}$ a
unique (positive)  measure $\varrho_h$ on $\R$  with
$$\langle h, (L - z)^{-1} h\rangle = \int_\R \frac{1}{t - z}
d\varrho_h (t)$$ for all $z\in \C\setminus \R$. The total mass
$\varrho_h (\R) $ is bounded by $\|h\|^2$. The measure $\varrho_h$
is known as \textit{spectral measure} (of $h$ with respect to $L$).
Moreover, for any measurable $\Psi : \R\longrightarrow \C$ there
exists a unique normal  operator $\Psi (L)$ on $\mathcal{H}$ with
domain $D (\Psi (L))$  satisfying
$$ D (\Psi (L)) =  \{h\in \mathcal{H} : \Psi \in L^2
(\R, \varrho_h)\}  \; \: \mbox{and} \; \:  \langle h,\Psi (L)
h\rangle = \int \Psi (s) d\varrho_h (s)$$ for all $h\in D(\Psi
(L))$. If $\Psi : \R\longrightarrow \R, x\mapsto x^n$, is the
$n$-the power, then we write $L^n$ instead of $\Psi (L)$.

By polarization, we can also  assign to $f,g\in \mathcal{H}$ the
signed measure $\varrho_{f,g}$ on $\R$ defined via
$$\varrho_{f,g}:= \frac{1}{4} \sum_{k =0}^{3} i^k \varrho_{f+ i^k
g}.$$ Then, a direct  calculation shows
\begin{eqnarray*}\int \Psi d
\varrho_{f,g} & = & \frac{1}{4}\sum_{k =0}^{3} i^k \int \Psi
d \varrho_{f+ i^k g}\\
 & =&\frac{1}{4}\sum_{k =0}^{3} i^k \langle (f + i^k g), \Psi (L)( f +
 i^k g)\rangle\\
  &= &  \langle f, \Psi (L) g\rangle
\end{eqnarray*} for all
measurable $\Psi : \R \longrightarrow \C$ with $f,g\in D(\Psi (L))$.
From this a rather straightforward  estimate gives
\begin{eqnarray*}
\left| \int \Psi (s) d\varrho_{f,g} (s)\right| & = &  \left|\langle
f, \Psi (L) g\rangle \right|\\
&\leq & \langle f, |\Psi| (L) f\rangle ^{1/2} \langle g, |\Psi| (L)g
\rangle^{1/2}\\
&\leq & \frac{1}{2} \left(\langle f, |\Psi| (L) f\rangle +\langle g,
|\Psi| (L)g \rangle \right).
\end{eqnarray*}
Similar statements hold in real Hilbert space. We refrain from
giving  details.

\smallskip

After these preparations we can now come to our  main technical
result. This result is just a combination of the Taylor theorem and
the spectral theorem.

\begin{proposition}[Taylor expansion and spectral calculus] \label{prop-taylor} Let $L$ be a selfadjoint
operator in the complex  Hilbert space $\mathcal{H}$. Let $N$ be a
natural number and $f,g\in D (L^{N+1})$ be given. Then, the estimate
$$\left| \langle f, \varPhi (L) g\rangle - \sum_{n=0}^N
\frac{\varPhi^{(n)} (0)}{n!} \langle f, L^n g\rangle \right| \leq
\frac{\|\varPhi^{(N+1)}\|_\infty}{(N+1)!} \frac{ \langle f,
|L|^{N+1} f\rangle + \langle g, |L|^{N+1} g\rangle }{2} $$ holds for
any uniformly  bounded $\varPhi : \R\longrightarrow \C$, which is
$(N+1)$-times continuously differentiable with uniformly bounded
$(N+1)$-derivative.
\end{proposition}
\begin{proof}
By the assumption on $f,g$ we  obtain in particular
$$ \int s^n d\varrho_{f,g} (s) = \langle f, L^n g\rangle
$$ for all $n=0,1\ldots, N+1$.
Considering the  Taylor expansion of $\varPhi$ around $0$  we infer
$$\varPhi (s) = \sum_{n=0}^N
\frac{\varPhi^{(n)} (0)}{n!} s^n +  \frac{s^{N+1}}{(N+1)!} R_{N+1}
(s)$$ with
$$|R_{N+1} (s)| \leq \|\varPhi^{(N+1)}\|_\infty.$$
Invoking the considerations on the spectral theorem preceding this
proposition, we then  infer
\begin{eqnarray*} \langle f, \varPhi (L)
g\rangle  &= & \int \varPhi (s) d\varrho_{f,g} (s)\\
&=& \sum_{n=0}^\N \frac{\varPhi^{(n)} (0)}{n!} \langle f, L^n
g\rangle + \frac{1}{(N+1)!} \int s^{N+1} R_{N+1} (s) d\varrho_{f,g}
(s).
\end{eqnarray*}
It remains to estimate the remainder term $$T:= \int s^{N+1} R_{N+1}
(s) d\varrho_{f,g} (s).$$  By  $$|s^{N+1} R_{N+1} (s)| \leq
|s|^{N+1} \| \varPhi^{(N+1)}\|_\infty$$  and the considerations on
the spectral theorem above, we clearly have
$$ |T| \leq \|\varPhi^{(N+1)}\|_\infty \frac{ \langle f,
|L|^{N+1} f\rangle + \langle g, |L|^{N+1} g\rangle }{2}$$ and the
proof of the proposition is finished.
\end{proof}

From the previous proposition we can rather directly obtain our main
abstract results on short time behavior of semigroups and unitary
groups associated to selfadjoint operators.

In order to phrase our results appropriately it is convenient to
introduce one  further piece of notation:  Whenever $L$ is  a
selfadjoint operator on the Hilbert space $\mathcal{H}$ and  $f,g\in
\mathcal{H}$ are given with $f,g\in D (L^n)$ for all natural numbers
$n$  we define
$$d_L (f,g):=\inf\{n : \langle f, L^n g\rangle \neq 0\},$$ where the
infimum over the empty set is taken to be $\infty$.

\begin{lemma}[Short time behavior -  semigroup] \label{lem:stbs} Let $L$ be a selfadjoint operator on the  Hilbert
space $\mathcal{H}$ with $L\geq 0$. Let $f,g\in \mathcal{H}$ with
$f,g\in D (L^n)$ for all natural numbers $n$ be given. Then, the
estimate
$$ \left| \langle f, e^{-t L} g\rangle - (-t)^n \cdot \frac{\langle f, L^n g\rangle }{n!}\right|
\leq t^{n+1} \cdot  \frac{\langle f, L^{n+1} f\rangle + \langle g,
L^{n+1} g\rangle }{2 \cdot (n+1)!} $$ holds  for all $n\leq d_L
(f,g)$ and all $t\geq 0$.
\end{lemma}
\begin{remark} Let us note that the statement of the lemma covers
two cases  at once. If  $d_L (f,g) < \infty$ it suffices to consider
$n = d_L (f,g)$ and this statement then gives the statement for the
smaller values of $n$ as well. If $d_L (f,g) =\infty$ then each  $n$
gives valid information.
\end{remark}

\begin{proof} This follows  from the previous proposition
with $L$ replaced by $t L$ and
 $\varPhi$ a suitable  extension of $[0,\infty)\longrightarrow
[0,\infty),\:  x\mapsto e^{-x}$, to  $\R$. Note that the actual
values of the extension on $(-\infty, 0)$ are completely irrelevant
due to $L\geq 0$. Note also that $|L| = L\geq 0$.
\end{proof}

\begin{lemma}[Short time behavior - unitary group]\label{lem:stbu} Let $L$ be a selfadjoint operator on the complex  Hilbert
space $\mathcal{H}$. Let $f,g\in \mathcal{H}$ with $f,g\in D (L^n)$
for all natural numbers $n$ be given. Then, the estimate
$$ \left| \langle f, e^{-i t L} g\rangle - (-i t)^n  \cdot \frac{\langle f, L^n g\rangle}{n!}\right|
\leq t^{n+1} \cdot  \frac{\langle f, |L|^{n+1} f\rangle + \langle g,
|L|^{n+1} g\rangle }{2 \cdot (n+1)!} $$ holds  for all $n\leq d_L
(f,g)$ and all $t\geq 0$.
\end{lemma}
\begin{proof} This follows immediately from the previous proposition
with $L$ replaced by $t L$ and
 $\varPhi : \R \longrightarrow
\C,\:  x\mapsto e^{-i x}$.
\end{proof}

\begin{corollary}[Leading exponent] \label{cor:leading-exponent} Let $L$ be a selfadjoint operator on the  Hilbert
space $\mathcal{H}$ with $L\geq 0$. Let $f,g\in \mathcal{H}$ with
$f,g\in D (L^n)$ for all natural numbers $n$ be given. The
following statements hold:
\begin{itemize}
  \item [(a)]  If $d_L (f,g) <\infty$ holds, we have
$$ \lim_{t\to 0+} \frac{\log|\langle f , e^{-t L} g\rangle |}{\log
t} = d_L (f,g) =   \lim_{t\to 0+} \frac{\log|\langle f , e^{-i t L}
g\rangle |}{\log t}.$$

  \item [(b)] If $d_L (f,g) = \infty$ holds, there exists for any natural
number $n$ a constant  $C_n (f,g) > 0  $ with
$$ |\langle f, e^{-t L} g\rangle |, |\langle f, e^{-i t L} g\rangle | \leq C_n (f,g) \cdot  t^{n+1}$$
for all $t\geq 0$.
\end{itemize}
\end{corollary}

\begin{remark} As mentioned already our results can also be adapted
to real Hilbert space. This can be used to obtain complete analogues
to the results above  for the short term behavior of the semigroup
group $e^{- t L}$ of an arbitrary selfadjoint operator $L$ with
$L\geq 0$ on a real Hilbert space.
\end{remark}

\section{Application to graphs} \label{Application-to-graphs}
In this section we apply the abstract results of the previous
section to graphs. In this situation one can give rather concrete
interpretations of the terms in question.

\bigskip

We consider a discrete set $X$ and  denote the set of all functions
on $X$ by $C(X)$ and the set of all functions with finite support by
$C_c (X)$. For $x\in X$ we denote the characteristic function of $x$
by $1_x$.  We can then think of a    map $m : X\longrightarrow
(0,\infty)$ as a measure (of full support). In this way,  $(X,m)$
becomes a measure space. In particular, there is a natural Hilbert
space $\ell^2 (X,m)$ associated to $(X,m)$.

\smallskip

A \textit{combinatorial graph over $X$} is given by a set $E\subseteq X\times
X$ with
\begin{itemize}
\item  $(x,y) \in E\Longrightarrow (y,x)\in E$,

\item $(x,x)\notin E$ for all $x\in X$,
\end{itemize}
and which we call the \emph{edge set} of the graph.

Whenever $E$ is the edge set of a combinatorial graph over $X$ a sequence
$(x_1,\ldots, x_n)$  in $X$ with $(x_j,x_{j+1})\in E$ for
$j=1,\ldots, n-1$,  is called \textit{path of length $n$ from $x_1$
to $x_n$}. If between any points of $X$ there exists a path, then
the graph $E$  is called \textit{connected}. In this case the
\textit{combinatorial graph distance} $d_E (x,y)$ between $x$ and
$y$ is defined to be the smallest length of a path connecting $x$
and $y$ if $x\neq y$ and to be $0$ if $x = y$.

\smallskip

A selfadjoint operator $L\geq 0$ on $\ell^2 (X,m)$  with $C_c (X)
\subseteq D(L^{1/2})$ is said to be \textit{a Laplacian based on a
graph with edge set $E$} if  for any $x,y\in X$ with $x \neq y$ we
have
$$ \langle 1_x, L 1_y\rangle = \langle
L^{1/2} 1_x, L^{1/2} 1_y \rangle \leq  0$$ with strict inequality if
and only if $(x,y)$ belongs to $E$.  This means that the
off-diagonal matrix elements of $L$ are non-positive. For example
Laplacians arising from Dirichlet forms on graphs (see below)
satisfy this property.

\begin{lemma} Let a countable set $X$ together with a measure  $m : X\longrightarrow
(0,\infty)$ be given and let  $E$ be the edge set of a connected combinatorial graph
over $X$.  Let  $L\geq 0 $ be  a Laplacian based on the graph with $C_c (X) \subset D(L^n)$ for all natural numbers $n$. Then,
$$ d_L (1_x, 1_y) = d_E (x,y)\;\:\mbox{and}\;\:  (-1)^{d_L (1_x, 1_y)} \langle 1_x, L^{d_L (x,y)}  1_y\rangle >0$$
hold for all $x,y\in X$.
\end{lemma}
\begin{proof} For $x,y\in X$ we define $e(x,y) := \langle 1_x, L
1_y\rangle$. Then,  a rather direct induction shows
$$ \langle 1_x, L^n 1_y\rangle = \sum_{x_1\in X} \ldots
\sum_{x_{n-1} \in X} \frac{ e(x,x_1) \ldots e(x_{n-1},
y)}{m(x_1)\ldots m (x_{n-1})}$$ for each $n\geq 2$,  where each of
the consecutive sums is absolutely convergent.  As $L$ is based on
the graph we have $e(x,y)\neq 0$ for $x\neq y$  if and only if
$(x,y)$ belongs to $E$ and in this case $e(x,y) < 0$ holds. From
this we easily infer the desired statement.
\end{proof}

With the previous lemma at hand, we can directly apply the
considerations of the previous section, in particular, Lemma
\ref{lem:stbs} and Lemma \ref{lem:stbu},  to obtain the following
result.

\begin{theorem} Let a countable  set $X$ together with a measure  $m : X\longrightarrow
(0,\infty)$ be given. Let $E$ be the edge set of a connected
combinatorial  graph over $X$ and let the operator  $L\geq 0 $ on
$\ell^2 (X,m)$ be a Laplacian based on the graph  with $C_c
(X)\subset D(L^n)$ for all natural numbers $n$.   Then, for all
$x,y\in X$ we have for all $t\geq 0$ the estimates
$$ \left| \langle 1_x, e^{-t L} 1_y\rangle - t^{d }  \cdot \frac{|\langle 1_x,
L^{d}  1_y\rangle| }{d!}  \right| \leq t^{d  +1} \cdot C(x,y)
$$ and
$$ \left| |\langle 1_x, e^{-i t L} 1_y\rangle| - t^{d } \cdot \frac{|\langle 1_x,
L^{d}  1_y\rangle| }{d!}  \right| \leq t^{d  +1} \cdot  C(x,y)$$
with
$$ d:= d_E (x,y), \;\: C(x,y):= \frac{\langle 1_x, L^{d +1} 1_x\rangle + \langle 1_y, L^{d
+1} 1_y\rangle }{2 \cdot (d+1)!},\:\; \mbox{and}\;\: |\langle 1_x,
L^d 1_y\rangle| >0.$$
\end{theorem}

From this theorem we directly obtain the following  corollary.

\begin{corollary} Consider the situation of the previous theorem.
Then, for all $x,y\in X$  we have
$$ \lim_{t\to 0+} \frac{\log|\langle 1_x  , e^{-t L} 1_y \rangle |}{\log
t} = d_E (x,y) =   \lim_{t\to 0+} \frac{\log|\langle 1_x , e^{-i t
L} 1_y\rangle |}{\log t}.$$
\end{corollary}

\medskip

We finish this section by discussing a specific \textbf{application}
of the previous results in the framework of regular Dirichlet forms
on graphs. For further details of this framework  we refer the
reader to   \cite{KL1}.

\smallskip

A pair $(b,c)$ is called a \textit{weighted graph over $X$} if $ b :
X\times X\longrightarrow [0,\infty)$ satisfies
\begin{itemize}
\item $b(x,y) = b(y,x)$,
\item $b(x,x) =0$,
\item $\sum_{z\in X} b(x,z) < \infty,$
\end{itemize}
for all $x,y\in X$ and  $c : X\longrightarrow [0,\infty)$ is
arbitrary. Each such graph induces  a combinatorial graph with edge set $E_{b,c}$
via
$$E_{b,c}:=\{ (x,y) \in X\times X : b(x,y) >0\}.$$
Moreover, any such graph $(b,c)$ comes with a form $$Q_{b,c} : C_c
(X)\times C_c (X)\longrightarrow \C$$ defined via
$$Q_{b,c} (f,g) =\frac{1}{2}\sum_{x,y} b(x,y) (f(x) - f(y))
\overline{( g(x) - g(y))} + \sum_{x\in X} c(x) f(x)
\overline{g(x)}.$$ This form $Q_{b,c}$ allows for  closed extensions
$Q$  satisfying $Q (f,f)\geq 0$ for all $f$ in the domain of $Q$.
Then, the selfadjoint operator $L$ associated to such a closed
extension  $Q$ will be a Laplacian over $E_{b,c}$ by construction.
This operator $L$ may or may not satisfy the condition that $C_c (X)
\subset D(L^n)$ for all natural numbers $n$. There are, however, two
instances in which this is automatically the case. These are
presented next.

\medskip

\noindent \textbf{Example - graphs with  bounded degree.} If the
degree
$$\Deg : X\longrightarrow [0,\infty), \Deg (x) = \frac{1}{m(x)}
\left(\sum_{z\in X} b(x,z) + c(x)\right)$$ is  bounded, then the
operator $L$ is bounded on $\ell^2 (X,m)$, see e.g. \cite{HKLW}.
Hence, $C_c (X) \subset D (L^n)$ holds for all natural numbers $n$.

\medskip

\noindent \textbf{Example - locally finite graphs.} Here, we follow
\cite{KL1} to which we refer for details and proofs.  If $Q$ is a
closed restriction of the form $Q^{\max}$ with domain $D( Q^{\max})$
given by those $f\in \ell^2 (X,m)$  with $$ \widetilde{Q}
(f,f):=\frac{1}{2}\sum_{x,y} b(x,y) (f(x) - f(y)) \overline{( f(x) -
f(y))} + \sum_{x\in X} c(x) f(x) \overline{f(x)} <\infty$$ and
$$Q^{\max} (f,f) := \widetilde{Q} (f,f)$$
then $L$ is a restriction of the operator $\widetilde{L}$ acting on
$$\widetilde{F} :=\{f\in C(X) : \sum_{y\in X} b(x,y) |f(y)|  <\infty\mbox{ for all }x\in X\}$$
via
$$\widetilde{L} f(x) :=\frac{1}{m(x)} \left(\sum_{y\in X} b(x,y) (f(x) -
f(y)) + c(x) f(x)\right).$$

If  $b$ is locally finite (i.e., for each $x\in X$ the number of
elements of $\{y \in X : b(x,y) >0\}$ is finite) then $C_c (X)$ can
be shown to belong to the domain  of definition of $D(L)$. Moreover,
$C_c (X)$ is clearly invariant under $\widetilde{L}$ and, hence,
$C_c (X)\subset D(L^n)$ holds for all natural numbers $n$.

\smallskip

Thus, in both of these situations the previous results apply and we
have the short time asymptotics  of the heat diffusion governed by
the combinatorial metric.

\section{Remark on an application to
manifolds}\label{Remark-on-application-to-manifolds}  In this
section we shortly remark on a simple application of our main
abstract result on manifolds.

\smallskip

Let $M$ be a Riemannian manifold  with volume $m$. Let $L$ be an
arbitrary selfadjoint operator on $L^2 (X, m)$  with $L\geq 0$ and
the property that it maps $C_c^\infty (X)$ into itself and that the
support of $Lf $ is contained in the support of $f$ for $f\in
C_c^\infty (X)$. Then, from Corollary \ref{cor:leading-exponent} we
obtain for any $f,g\in C_c^\infty (X)$ with disjoint support and any
natural number $n$ the existence of $C_n (f,g) > 0  $ with
$$ |\langle f, e^{-t L} g\rangle |, |\langle f, e^{-i t L} g\rangle | \leq C_n (f,g) \cdot  t^{n+1}$$
for all $t\geq 0$.  This  result does not only  hold for the
semigroup and unitary group of the Laplace-Beltrami operator, where,
of course, much more is known. It also applies to, say, suitable
fourth order operators. Let us emphasize that it is  completely
independent of any additional geometric assumptions.


\end{document}